\documentclass[12pt,leqno]{amsart}
\usepackage{amsmath,amsfonts,amssymb,amscd,amsthm,amsbsy,upref}

\newtheorem*{rem}{Remark}

\usepackage{ifthen}

\newcounter{mylisti} \newcounter{mylistii}
\newcounter{nest}
\newcommand{\defaultlabel}{}

\newenvironment{mylist}[1]{%
\addtocounter{nest}{1}
\ifthenelse{\value{nest}=1}{%
  \renewcommand{\defaultlabel}{(\roman{mylisti})\hfill}}{%
  \renewcommand{\defaultlabel}{(\alph{mylistii})\hfill}}
\begin{list}{\defaultlabel}{%
    \ifthenelse{\value{nest}=1}{\usecounter{mylisti}}{%
      \usecounter{mylistii}}
    
    \addtolength{\itemsep}{0.5ex}
    \settowidth{\labelwidth}{#1}
    \setlength{\leftmargin}{\labelwidth}
    \addtolength{\leftmargin}{\labelsep}}}{\addtocounter{nest}{-1}
\end{list}}

\newcommand{\bn}{\ensuremath{\mathbb N}}

\newcommand{\cA}{\ensuremath{\mathcal A}}
\newcommand{\cB}{\ensuremath{\mathcal B}}
\newcommand{\cC}{\ensuremath{\mathcal C}}

\newcommand{\cJ}{\ensuremath{\mathcal J}}

\newcommand{\cL}{\ensuremath{\mathcal L}}

\newcommand{\cT}{\ensuremath{\mathcal T}}

\newcommand{\dist}{\ensuremath{\mathrm{dist}}}

\newcommand{\norm}[1]{\lVert #1\rVert}
\newcommand{\bignorm}[1]{\big\lVert #1\big\rVert}

\newcommand{\restrict}{\ensuremath{\!\!\restriction}}

\newcommand{\co}{\mathrm{c}_0}

\renewcommand{\geq}{\geqslant}
\renewcommand{\leq}{\leqslant}

\newcommand{\ds}{\displaystyle}

\newcommand{\ie}{\textit{i.e.,}\ }

\newtheorem{thm}{Theorem}

\newtheorem{cor}[thm]{Corollary}

\begin{document}
\title[Closed subideals of operators]{The cardinality of the sublattice of closed ideals of operators between certain classical sequence spaces}	
\author{D.~Freeman, Th.~Schlumprecht and A.~Zs\'ak}


\thanks{The first author was supported by grant 353293 from the Simons Foundation. The second author's research was supported by
NSF grant DMS-1912897.}

\maketitle

\begin{abstract}
  Theorem~A and Theorem~B of~\cite{FSZ}  state that for
 $1<p<\infty$ the lattice of closed ideals of $\cL(\ell_p,\co)$,
 $\cL(\ell_p,\ell_\infty)$ and of $\cL(\ell_1,\ell_p)$ are at least
 of cardinality $2^{\omega}$. Here we show that
 the cardinality of the lattice of closed ideals of
 $\cL(\ell_p,\co)$, $\cL(\ell_p,\ell_\infty)$ and of
 $\cL(\ell_1,\ell_p)$, is at least $2^{2^\omega}$, and thus  equal to
 it.
\end{abstract}

In~\cite{FSZ} we construct $2^{\omega}$ operators from $\ell_p$ to
$\co$ which generate distinct closed operator ideals in
$\cL(\ell_p,\co)$. Here we show that we can naturally choose a
subset of those operators of size $2^{\omega}$ such that not only does
each operator generate a distinct closed operator ideal, but each
subset of these operators also generates a distinct closed operator
ideal. Hence there are in fact $2^{2^{\omega}}$ distinct closed
operator ideals in $\cL(\ell_p,\co)$.

Let $1\leq p<\infty$. For appropriately chosen sequences $(u_n)$
and $(v_n)$ in $\bn$ we constructed a uniformly bounded sequence
$(T_n)$ of operators $T_n\colon\ell_{2}^{u_n}\to \ell_\infty ^{v_n}$
given by appropriately scaled RIP matrices, and defined for $M\subset
\bn$ the operator
\[
T_M\colon U=(\oplus_{n\in\bn}\ell_2^{u_n})_{\ell_p}\to
V=(\oplus_{n\in\bn}\ell_\infty^{v_n})_{\co},\quad (x_n)\mapsto
Q_M(T_n(x_n):n\in\bn)\ ,
\]
where $Q_M\colon V\to V$ is the canonical projection onto the
coordinates in $M$.

For Banach spaces $X,Y,W,Z$, and for a set of operators
$\cT\subset\cL(W,Z)$, we let $\cJ^{\cT}(X,Y)$ be the closed ideal of
$\cL(X,Y)$ generated by $\cT$. Thus $\cJ^{\cT}(X,Y)$ is the closure in
$\cL(X,Y)$ of
\[
\Big\{ \sum_{j=1}^n A_j S_j B_j:\,n\in\bn,\ 
(S_j)_{j=1}^n\subset\cT,\ (A_j)_{j=1}^n\subset \cL(Z,Y),\ 
(B_j)_{j=1}^n\subset\cL(X,W)\Big\}\ .
\]
We write $\cJ^T(X,Y)$ instead of $\cJ^{\cT}(X,Y)$ if $\cT=\{T\}$.

For infinite subsets $M,N$ of
$\bn$, \cite[Theorem 1]{FSZ}  states that:
\begin{mylist}{(ii)}
\item
 If $M\setminus N$ is infinite then $T_M\not\in \cJ^{T_N}(U,V)$,
\item
 if  $N\setminus M$ is finite then $\cJ^{T_N}\subset \cJ^{T_M}(U,V)$.
\end{mylist}
Moreover, the proof of~\cite[Theorem~1]{FSZ} shows that if
$M\setminus N$ is infinite, then there is a functional
$\Phi\in\cL(U,V)^*$, $\norm{\Phi}\leq 1$, with $\Phi(T_M)=1$ and 
$\Phi\restrict_{\cJ^{T_N}(U,V)}=0$, which means that
$\dist(T_M,\cJ^{T_N}(U,V))\geq 1$. It follows from the proof
of~\cite[Theorem~6]{FSZ}(see also the subsequent remark) that if $W$
is a Banach space containing 
$\co$ and $J\colon V\to W$ is an isomorphic embedding, then we also
have $\dist(J\circ T_M,\cJ^{T_N}(U,W))\geq c$, where
$c=\norm{J^{-1}}^{-1}$.

Using now the same approach as in~\cite{JS},  we can easily deduce the
following corollary.
\begin{cor}
 Let $W$ be a Banach space containing $\co$. Then the cardinality of
 the lattice of closed ideals of $\cL(U,W)$ is at least $2^{2^\omega}$. In
 particular, for $1<p<\infty$, the cardinality of the lattice of
 closed ideals of $\cL(\ell_p,\co)$ and of
 $\cL(\ell_p,\ell_\infty)$ is $2^{2^\omega}$.
\end{cor}
\begin{proof}
 Since $V\cong\co$ and $U\sim\ell_p$ when $p>1$, we need only to
 prove the first statement.
 Let $\cC$ be a family of infinite subsets of $\bn$ whose cardinality
 is  $2^\omega$ and whose elements are pairwise almost disjoint, \ie
 if $M\not=N$ are in $\cC$ then $N\cap M$ is finite. For
 $\cA\subset\cC$ we put $\cT(\cA)=\{ T_M: M\in \cA\}$ and
 $\cJ(\cA)=\cJ^{\cT(\cA)}(U,W)$. We claim that for any  two
 nonempty subsets $\cA$ and $\cB$ of $\cC$, we have 
 $\cJ(\cA)\neq \cJ(\cB)$. Indeed, without loss of generality we can
 assume 
 that $\cA\setminus\cB\neq\emptyset$. Let $M\in\cA\setminus\cB$ and
 let $J\colon V\to W$ be an isomorphic embedding. We show that
 $d(J\circ T_M,\cJ(\cB))\geq c$ where $c=\norm{J^{-1}}^{-1}$. Since
 $J\circ T_M\in\cJ(\cA)$, this shows that $\cJ(\cA)\neq \cJ(\cB)$.

 Let  $n\in\bn$,  $(N_j)_{j=1}^n\subset \cB$,
 $(A_j)_{j=1}^n\subset\cL(V,W)$ and $(B_j)_{j=1}^n\subset\cL(U)$. Put
 $N=\bigcup_{j=1}^n N_j$. We have $A_jT_{N_j}B_j=A_jQ_{N_j}T_NB_j$
 for $j=1,2,\dots,n$, and hence
 $\sum_{j=1}^nA_jT_{N_j}B_j\in\cJ^{T_N}(U,W)$. Since $M\setminus N$
 is infinite, it follows that
 \[
 \bignorm{J\circ T_M-\sum_{j=1}^nA_jT_{N_j}B_j}\geq
 \dist(J\circ T_M,\cJ^{T_N}(U,W))\geq c\ .
 \]
 Since $\cJ(\cB)$ is the closure of the set of operators of the form
 $\sum_{j=1}^nA_jT_{N_j}B_j$, the proof is complete.
\end{proof}
\begin{rem}
 A very simple duality argument (see~\cite[Proposition~7]{FSZ}
 and~\cite[Theorem~8]{FSZ}) shows that for $1<q<\infty$, the lattice
 of closed ideals of $\cL(\ell_1,\ell_q)$ is also of cardinality
 $2^{2^\omega}$. The same is true in
 $\cL\big(\ell_1,(\bigoplus_{n\in\bn}\ell_2^n)_{\co}\big)$.
\end{rem}
In~\cite{JS} it was shown that the cardinality of the set closed
ideals of $\cL(L_p)$, $1<p<\infty$, is $2^{2^\omega}$. Note that the
Hardy space $H_1$ and its predual VMO can be seen as the ``well
behaved'' limit cases of the $L_p$-spaces. For example $\ell_2$ is
complemented in both spaces, and $H_1$ contains a complemented copy of
$\ell_1$ and VMO a complemented copy of $\co$ (\textit{cf.}~\cite{M1}
and~\cite[page~125]{M2}), and thus we deduce the following corollary.
\begin{cor}
 The  cardinality of the lattice of closed ideals of
 $\cL(\text{VMO})$ and $\cL(H_1)$ is $2^{2^{\omega}}$.
\end{cor}

\address{D.~Freeman,
Department of Mathematics and Statistics,
St Louis University,
St Louis, MO 63103\,USA}.

\email{\tt daniel.freeman@slu.edu}

\address{Th.~Schlumprecht,
Department of Mathematics,
Texas A\&M University,
College Station, TX 77843,
USA and
Faculty of Electrical Engineering, 
Czech Technical University in Prague
Zikova 4, 166 27, Prague,
Czech Republic.

\email{\tt t-schlumprecht@tamu.edu}}

\address{A.~Zs\'ak,
Peterhouse, Cambridge, CB2 1RD,
United Kingdom.

\email{\tt a.zsak@dpmms.cam.ac.uk}
\end{document}